\documentclass[12pt]{article}
\usepackage{amsmath, amsthm, amscd, amsfonts, amssymb, graphicx}
\usepackage{titles}
\usepackage{algorithm}
\usepackage{algpseudocode}
\usepackage{subfigure} 
\usepackage{caption}
\usepackage{graphics,graphicx}
\usepackage{amsmath, amssymb, amsthm}
\usepackage{mathrsfs} 
\usepackage{enumitem}

\usepackage{indentfirst}

\numberwithin{equation}{section}
\usepackage{epsfig}
\usepackage{color}
\usepackage{fancyhdr}
\usepackage{titlesec}
\usepackage{tikz}
\usepackage{afterpage}
\usetikzlibrary{arrows}
\usetikzlibrary{decorations.markings}
\tikzstyle{block}=[draw opacity=0.7,line width=1.4cm]
\tikzset{
	big black arrow/.style={
		decoration={markings,mark=at position 1 with {\arrow[scale=2.5,black]{>}}},
		postaction={decorate},
		shorten >=0.4pt},
	line/.style={draw, ->}}

\usepackage{eufrak}
\newtheorem{theorem}{Theorem}

\newtheorem{definition}{Definition}

\newtheorem{corollary}[theorem]{Corollary}
\theoremstyle{definition}

\theoremstyle{remark}

\numberwithin{equation}{section}

\allowdisplaybreaks

\title{\textbf{Relations for partitions with distinct even parts except the largest part which is even}}
\date{}
\newcommand{\qPoch}[3]{(#1;#2)_{#3}}
\begin{document}
\maketitle
\vspace{-2cm}
\maketitle

\begin{center}
	{\bf Gaurab Bardhan$^1$ and Nipen Saikia$^{2, \ast}$}\\
	$^1$Department of Mathematics, Tyagbir Hem Baruah College,\\ Jamugurihat, Sonitpur, Assam, India.\\
	E. Mail: gaurabbardhan561@gmail.com
	\vskip2mm
	$^2$Department of Mathematics, Rajiv Gandhi
	University,\\ Rono Hills, Doimukh, Arunachal Pradesh, India.\\
	E. Mail(s): nipennak@yahoo.com\\
	$^\ast$\textit{Corresponding author}.\end{center}
\begin{abstract}
In this paper,  we prove some new \(q\)-series identities connecting \(4\)-regular partitions and partitions with distinct even parts with largest part being odd. We also define three new partition functions with distinct even parts except the largest part which is even,  and prove identities connecting the three partitions with \(4\)-regular partitions. Moreover, we also offer some congruence for the three newly defined partitions.
\end{abstract} 


 \section{Introduction}
For any positive integer \(n\)  and any complex numbers $\alpha$ and \(q\) with \(|q|<1\), define  \(q\)-series in standard notations as
\begin{equation}\label{ie1}
   \qPoch{\alpha}{q}{0}=1,\quad \qPoch{\alpha}{q}{n}=\prod_{j=0}^{n-1}(1-\alpha q^{j}),\quad \qPoch{\alpha}{q}{\infty}=\prod_{j=0}^{\infty}(1-\alpha q^{j}). 
\end{equation}
Further for non-negative integers \(m\) and \(n\), throughout this paper we will use the following basic facts about \(q\)-series \cite{1,7}: 
\begin{equation}\label{ie2}
  \qPoch{\alpha}{q}{n+m}=\qPoch{\alpha}{q}{m}\qPoch{\alpha q^m}{q}{n},\quad \qPoch{\alpha }{q}{\infty}=\qPoch{\alpha}{q}{n}\qPoch{\alpha q^n}{q}{\infty},\quad \qPoch{\alpha}{q}{\infty}=\qPoch{\alpha}{q^2}{\infty}\qPoch{\alpha q}{q^2}{\infty} 
\end{equation}
The \(q\)-binomial theorem \cite{1} is given by
\begin{equation}\label{ie3}
  \sum_{n=0}^{\infty}\dfrac{\qPoch{\alpha}{q}{n}}{\qPoch{q}{q}{n}}z^{n}=\dfrac{\qPoch{\alpha z}{q}{\infty}}{\qPoch{z}{q}{\infty}}.
\end{equation}
Also, for any complex number $\beta$, the following formula of Andrews, Subbarao and Vidyasagar \cite{3} will be useful in this paper,
\begin{equation}\label{ie4}
\sum_{n>0} \dfrac{\qPoch{\alpha}{q}{n}}{\qPoch{\beta}{q}{n}} q^n = \dfrac{q\qPoch{\alpha}{q}{\infty}}{\beta\qPoch{\beta}{q}{\infty}(1-\dfrac{\alpha q}{\beta})}+\dfrac{1-\dfrac{q}{\beta}}{1-\dfrac{\alpha q}{\beta}}. 
\end{equation}
Recall that,  for any positive integer \(t\), a \(t\)-regular partition of a positive integer $n$ is a partition which parts are not divisible by \(t\). For example, the $4$-regular partitions of 5 are given $5, 3+2, 3+1+1, 2+2+1, 2+1+1+1$, and $1+1+1+1+1$. If \(reg_t(n)\) denotes the number of \(t\)-regular partitions of \(n\), then its generating function is given by
\begin{equation}\label{regg}
  \sum_{n=0}^{\infty}reg_t(n)q^{n}=\dfrac{\qPoch{q^t}{q^t}{\infty}}{\qPoch{q}{q}{\infty}}.  
\end{equation}
Let \(ped(n)\) denote the number of partitions of \(n\) with distinct even parts. Then the generating function of \(ped(n)\) is given by
\begin{equation}\label{pedg}
  \sum_{n=0}^{\infty}ped(n)q^{n}=\dfrac{\qPoch{-q^2}{q^2}{\infty}}{\qPoch{q}{q^2}{\infty}}.  
\end{equation}For example, \(ped(5)=6\) with the relevant partitions as $5, 4+1, 3+2, 3+1+1, 2+1+1+1$ and $1+1+1+1+1$.
In recent years,  \(ped(n)\) has been extensively studied for its arithmetic properties. One may refer \cite{2, 4,6,11,12} and references therein for details. From  \cite{2}, we note that
\begin{equation}\label{ie5}
  \sum_{n=0}^{\infty}ped(n)q^{n}=\dfrac{\qPoch{-q^2}{q^2}{\infty}}{\qPoch{q}{q^2}{\infty}}=\dfrac{\qPoch{q^4}{q^4}{\infty}}{\qPoch{q}{q}{\infty}}=\sum_{n=0}^{\infty}reg_4(n)q^{n},  
\end{equation}
which is a consequence of the following well known  identity of Lebesgue,
\begin{equation}\label{ie6}
\sum_{n=0}^{\infty} \dfrac{\qPoch{-q}{q}{n}}{\qPoch{q}{q}{n}} q^{(n+1)/2} = \qPoch{-q^2}{q^2}{\infty} \qPoch{-q}{q}{\infty} = \dfrac{\qPoch{q^4}{q^4}{\infty}}{\qPoch{q}{q}{\infty}}.
\end{equation}
 Another interesting partition function is the cubic partition function denoted by \(a(n)\) and first studied by Chan \cite{chan1}. The generating function of $a(n)$ denotes the number of cubic partitions. Then,
\begin{equation}\label{ce1}
    \sum_{n=0}^{\infty}a(n)q^{n}=\dfrac{1}{\qPoch{q}{q}{\infty}\qPoch{q^2}{q^2}{\infty}}.
\end{equation}
For congruences and other arithmetic properties of \(a(n)\) see \cite{chan3, chan1, chan2, chan4}.

In 2025, Andrews and Bachraoui \cite{DE} derived some relations connecting partition functions with distinct even parts $ped(n)$ and \(4\)-regular partition function $reg_4(n)$.  Further, they defined  three partition functions,  \(\operatorname{DE1}(n)\) that counts the number of partitions of \(n\) in which even parts are distinct and the largest part is odd;  \(\operatorname{DE2}(n)\) that counts the number of partitions of \(n\) in which even parts are distinct, the largest part is odd and appears atleast twice; and  \(\operatorname{DE3}(n)\) that counts the number of partitions of \(n\) in which even parts are distinct, the largest part is odd and appears exactly once. For example, \(\operatorname{DE1}(7)=7\) with the partitions given by \(7\), \(5+2\), \(5+1+1\), \(3+3+1\), \(3+2+1+1\), \(3+1+1+1+1\), \(1+1+1+1+1+1+1+1\); \(\operatorname{DE2}(7)=2\) with the partitions given by \(3+3+1\), \(1+1+1+1+1+1+1+1\); and \(\operatorname{DE3}(7)=5\) with the partitions given by \(
7\), \(5+2\), \(5+1+1\), \(3+2+1+1\), \(3+1+1+1+1+1.\)
The corresponding generating functions \cite{DE} are defined as,
\begin{equation}\label{ie7}
\sum_{n=0}^{\infty}\operatorname{DE1}(n)q^{n}=\sum_{n=0}^{\infty}\dfrac{\qPoch{-q^{2}}{q^{2}}{n}q^{2n+1}}{\qPoch{q}{q^{2}}{n+1}},
\end{equation}
\begin{equation}\label{ie8}
\sum_{n=0}^{\infty}\operatorname{DE2}(n)q^{n}=\sum_{n=0}^{\infty}\dfrac{\qPoch{-q^{2}}{q^{2}}{n}q^{4n+2}}{\qPoch{q}{q^{2}}{n+1}}
\end{equation} and
\begin{equation}\label{ie9}
\sum_{n=0}^{\infty}\operatorname{DE3}(n)q^{n}=\sum_{n=0}^{\infty}\dfrac{\qPoch{-q^{2}}{q^{2}}{n}q^{2n+1}}{\qPoch{q}{q^{2}}{n}}.
\end{equation}
Andrews and Bachraoui \cite[Corollary 1, 2, 3 and 4]{DE} further derived following partition relations, respectively:\\
(i) For \(n>0\), \(\operatorname{DE1}(n)+\operatorname{DE1}(n-1)\) equals the number of the $4$-regular partitions of \(n\).\\
(ii) For \(n>0\), \(\operatorname{DE2}(n)+\operatorname{DE2}(n-3)\) equals the number of \(4\)-regular partitions of \(n\) into parts each \(>1\).\\
(iii) For \(n>1\), \(\operatorname{DE3}(n+2)+\operatorname{DE3}(n-1)\) equals the number of \(4\)-regular partitions of \(n\).\\
(iv) For \(n>1\), \(\operatorname{DE3}(n+2)+\operatorname{DE3}(n-1)=\operatorname{DE1}(n)+\operatorname{DE1}(n-1)\).\\

Motivated by the above work, in this paper we first generalise the partition functions \(\operatorname{DE2}(n)\)  and \(\operatorname{DE3}(n)\)  and prove new relations connecting them with the $4$-regular partition function. Then we  define three new partition functions with distinct even parts except the largest part which is even, and establish relations for the partition functions  by defining their generating functions. Some  congruence properties are also derived in the process. 

In Section 2, we first generalise the partition functions \(\operatorname{DE2}(n)\)  and \(\operatorname{DE3}(n)\)  and prove new relations connecting them with the $4$-regular partition function.  We also define  three new partition functions and  main results are presented. Proofs are given in Section 3. In Section 4, we give  concluding remarks.

\section{Definitions and Main Results}
We define the generalisation of the partition functions \(\operatorname{DE2}(n)\)  and \(\operatorname{DE3}(n)\) in Definitions \ref{d1} and \ref{d2}, respectively.
\begin{definition}\label{d1}
   Let \(DE_{\geq k}(n)\) counts the number of partitions of \(n\) in which even parts are distinct, the largest part is odd and appears atleast \(k\) times. Then 
\begin{equation}\label{s2e1}
\sum_{n=0}^{\infty}{DE_{\geq k}}(n)q^{n}=\sum_{n=0}^{\infty}\dfrac{\qPoch{-q^{2}}{q^{2}}{n}q^{(2n+1)k}}{\qPoch{q}{q^{2}}{n+1}}. 
\end{equation}
\end{definition}
For example, \(DE_{\geq3}(10)=2,\) counting \(3+3+1,\) \(1+1+1+1+1+1+1+1+1+1.\) 
Also note that \(DE_{\geq2}(n)=\mathrm{DE2}(n).\) 

\begin{definition}\label{d2}
   Let \(DE_{k}(n)\) counts the number of partitions of \(n\) in which even parts are distinct, the largest part is odd and appears exactly \(k\) times. Then 
\begin{equation}\label{s2e2}
\sum_{n=0}^{\infty}DE_k(n)q^{n}=\sum_{n=0}^{\infty}\dfrac{\qPoch{-q^{2}}{q^{2}}{n}q^{(2n+1)k}}{\qPoch{q}{q^{2}}{n}}.
\end{equation}
\end{definition}
For example, \(DE_{2}(10)=3,\) counting \(3+3+2+1+1,\) \(3+3+1+1+1+1,\) \(5+5.\)
Also note that \(DE_{1}(n)=\mathrm{DE3}(n).\)

Next, we define three new partition functions that count the number of partitions of \(n\) such that all even parts are distinct except the largest part which is even. The corresponding generating functions are also defined.
\begin{definition}\label{d3}
	Let $DE_e(n)$ denote the number of partitions of $n$ in which no even part is repeated except the largest part which is even and unrestricted. Then
	\begin{equation}\label{d3e}
		\sum_{n \geq 0} DE_e(n) q^n = \sum_{n \geq 0} \dfrac{\qPoch{-q^2}{q^2}{n}q^{2n}}{\qPoch{q}{q^2}{n}}.
	\end{equation}
\end{definition}
For example, \(DE_e(6)=6,\) counting \(6,\) \(4+2,\) \(4+1+1,\) \(2+2+2,\) \(2+2+1+1,\) \(2+1+1+1+1.\)

\begin{definition}\label{d4}
	Let $DE_{ek}(n)$ denote the number of partitions of $n$ in which no even part is repeated and the largest part is even that appears exactly \(k\) times. Then
	\begin{equation}\label{d4e}
		\sum_{n \geq 0} DE_{ek}(n) q^n = \sum_{n \geq 0} \dfrac{\qPoch{-q^2}{q^2}{n}q^{(2n+2)k}}{\qPoch{q}{q^2}{n+1}}.
	\end{equation}   
\end{definition}
For example, \(DE_{e1}(6)=4,\) counting \(6,\) \(4+2,\) \(4+1+1,\) \(2+1+1+1+1.\)
\begin{definition}\label{d5}
	Let $DE_{e\geq k}(n)$ denote the number of partitions of $n$ in which no even part is repeated except the largest part which is even and appears atleast
	\(k\) times. Then, 
	\begin{equation}\label{d5e}
		\sum_{n \geq 0} DE_{e\geq k}(n) q^n = \sum_{n \geq 0} \dfrac{\qPoch{-q^2}{q^2}{n}q^{2nk}}{\qPoch{q}{q^2}{n}}.
	\end{equation} 
\end{definition}
For example, \(DE_{e\geq2}(6)=2,\) counting \(2+2+2,\) \(2+2+1+1.\) 

With the help of \eqref{s2e1} and \eqref{s2e2},  we derive following \(q\)-series identities for \(DE_{\geq k}(n)\) and \(DE_{k}(n)\) in Theorems 1 and 2, respectively:
\begin{theorem}\label{t1} We have
\begin{equation}\label{t1e}
   \sum_{m \geq 0}DE_{\geq k}(m) q^m=\dfrac{\qPoch{q^4}{q^4}{\infty}(q;q)_{k-1} }{\qPoch{q}{q}{\infty}} \sum_{n \geq 0}\dfrac{(q^{k};q^2)_n (q^{k+1};q^2)_nq^{2n}}{\qPoch{q^4}{q^4}{n}} - \dfrac{\qPoch{q^4}{q^4}{\infty}}{(1+q)\qPoch{q}{q}{\infty}}\cdot \dfrac{\qPoch{q^2}{q^2}{k-1}}{\qPoch{-q^3}{q^2}{k-1}}. 
\end{equation}
    
\end{theorem}
\begin{theorem}\label{t2} We have
\begin{equation}\label{t2e}
 \sum_{m\geq 0}DE_k(m)q^m
= \dfrac{q^k\qPoch{q^4}{q^4}{\infty}}{\qPoch{q}{q}{\infty}} \sum_{n\geq 0} \dfrac{\qPoch{q}{q}{2n+k-1} q^{4n}}{\qPoch{q^4}{q^4}{n}}-\dfrac{\qPoch{q^4}{q^4}{\infty}}{\qPoch{q}{q}{\infty}} \cdot \dfrac{q^{2k}\qPoch{q}{q}{\infty}}{(1-q)\qPoch{q^6}{q^4}{\infty}}\cdot \dfrac{\qPoch{-q^{3+2k}}{q^2}{\infty}}{\qPoch{q^{2k}}{q^2}{\infty}}.   
\end{equation}
\end{theorem}
Using Theorems \ref{t1}  and \ref{t2}, we derive relations of \(DE_{\geq 3}(n)\) and \(DE_{2}(n)\) with the \(4\)-regular partition function, which are presented in Theorem \ref{t3} and \ref{t4}, respectively below.
\begin{theorem}\label{t4}
    For \(n\geq12\), we have
$$
    DE_2(n-3)-DE_2(n-4)+DE_2(n-6)-DE_2(n-7)+DE_2(n-8)-DE_2(n-9)$$$$+DE_2(n-11)-
    DE_2(n-12)=reg_4(n-2)+reg_4(n-3)-2reg_4(n-4)-reg_4(n-5)$$\begin{equation}\label{t4e}+reg_4(n-6)+2reg_4(n-8)-reg_4(n-9)-2reg_4(n-10)+reg_4(n-11).
\end{equation}
\end{theorem}
\begin{theorem}\label{t3}
For \(n\geq10\), we have
$$
 DE_{\geq3}(n-1)-DE_{\geq3}(n-2)+DE_{\geq3}(n-4)-DE_{\geq3}(n-5)+DE_{\geq3}(n-6)-DE_{\geq3}(n-7)$$\begin{equation}\label{t3e}+
 DE_{\geq3}(n-9)-DE_{\geq3}(n-10)=-reg_4(n)-3reg_4(n-1)+3reg_4(n-2)-reg_4(n-5)+reg_4(n-6).
 \end{equation}
 \end{theorem}

Using the \(q\)-series identities derived by Andrews and Bachraoui in \cite{DE} and identities (\ref{t1e}), (\ref{t2e}), we derive  relations for $DE_e(n)$, $DE_{e1}(n)$, and $DE_{e\geq2}(n)$ which are presented  in Theorems \ref{t5}-\ref{t7} below:
\begin{theorem}\label{t5}
 For $n \geq 3$, we have
 \begin{equation}\label{t5e}
     DE_e(n) = \sum_{i=0}^{\lfloor \frac{n-1}{3}\rfloor} (-1)^ireg_4(n-1-3i), 
 \end{equation}where $\lfloor \cdot\rfloor$ denotes the floor function.
  \end{theorem}
\begin{theorem}\label{t6}
For \(n\geq2\), we have
\begin{equation}\label{t6e}
    DE_{e1}(n) = \sum_{i=0}^{n-2} (-1)^ireg_4(n-1-i).
\end{equation}
\end{theorem}
\begin{theorem}\label{t7}
     For \(n\geq2\), we have
\begin{equation}\label{t7e}
   DE_{e\geq2}(n-2)+DE_{\geq3}(n)=\sum_{i=0}^{n-1} (-1)^ireg_{4>1}(n-3i),
\end{equation}
where \(reg_{4>1}(n)\) counts the number of \(4\)-regular partitions of a non-negative integer \(n\) with all parts \(>1\).
\end{theorem}

Further, we prove some new congruences modulo \(2,\) \( 4,\) and \(8\) of \(DE_e(n)\), and \(DE_{e1}(n)\) which can be stated in the Theorems \ref{t8}-\ref{t15}. In the following theorems, 
let \(T_j = \dfrac{j(j+1)}{2}\) denote  triangular number and \(j\in\mathbb{N}\cup\{0\}\).

\begin{theorem}\label{t8} We have
\begin{enumerate}
    \item [\((i)\)] For \(n\geq T_j+2\) and \(n\neq2T_j\),  
 $$\sum_{j=0}^{\infty} DE_e(n-T_j+1)\equiv\sum_{j=0}^{\infty}DE_e(n-T_j-2)\pmod2.$$ 
    \item[\((ii)\)]  For \(n\geq T_j\) and \(n\neq2T_j\),  
     $$\sum_{j=0}^{\infty} DE_{e1}(n-T_j)\equiv\sum_{j=0}^{\infty}DE_{e1}(n-T_j+1)\pmod2.$$
\end{enumerate}
\end{theorem}

\begin{theorem}\label{t10} We have
\begin{enumerate}
    \item [\((i)\)] For \(n \geq 2j^2+2\) and \(n\neq T_j\), 
 $$\sum_{j=-\infty}^{\infty} DE_e(n-2j^2-2)\equiv\sum_{j=-\infty}^{\infty}DE_e(n-2j^2+1)\pmod2,$$
    \item[\((ii)\)]  For \(n \geq 2j^2,\) and \(n\neq T_j\),  $$ \sum_{j=-\infty}^{\infty} DE_{e1}(n-2j^2+1)\equiv\sum_{j=-\infty}^{\infty}DE_{e1}(n-2j^2)\pmod2.$$
\end{enumerate}
\end{theorem}
From Theorem \ref{t10} \((i)\), following congruences follow immeditely:
\begin{align*}
	&\text{For}\; n = 2, j = 0, \quad DE_{e}(0) + DE_{e}(3)\equiv0\pmod{2},\\
	&\text{For}\; n = 3, j = 0,\quad DE_{e}(1) + DE_{e}(4)\equiv 1 \pmod{2},\\
	& \text{For}\; n = 4, j = 0, 1,\quad DE_{e}(2) + DE_{e}(5) \equiv 0 \pmod{2},\\
	&\text{For}\; n = 5, j = 0, 1,\quad DE_{e}(3) + DE_{e}(6) \equiv 0\pmod{2},\\
	&\text{For}\; n = 6, j = 0, 1, \quad DE_{e}(4) + DE_{e}(7) \equiv 1\pmod{2},\\
	&\text{For}\; n = 7, j = 0, 1,\quad DE_{e}(5) + DE_{e}(8) \equiv 0 \pmod{2}.
\end{align*}
From Theorem \ref{t10} \((ii)\), following congruences follow immeditely:
\begin{align*}
&\text{For}\; n = 0, j = 0, \quad DE_{e1}(0) + DE_{e1}(1)\equiv1\pmod{2},\\
&\text{For}\; n = 1, j = 0, \quad DE_{e1}(1) + DE_{e1}(2)\equiv1\pmod{2},\\
&\text{For}\; n = 2, j = 0, 1, \quad DE_{e1}(2) + DE_{e1}(3)\equiv 0 \pmod{2},\\
& \text{For}\; n = 3, j = 0, 1, \quad DE_{e1}(3) + DE_{e1}(4) \equiv 1 \pmod{2},\\
&\text{For}\; n = 4, j = 0,  1\, \quad DE_{e1}(4) + DE_{e1}(5) \equiv 0 \pmod{2},\\
&\text{For}\; n = 5, j = 0,  1, \quad DE_{e1}(5) + DE_{e1}(6) \equiv 0 \pmod{2},\\
&\text{For}\; n = 6, j = 0, 1, \quad DE_{e1}(6) + DE_{e1}(7) \equiv 1 \pmod{2}.
\end{align*}
From the above congruences derived from Theorem \ref{t10},  we have the following corollary:
\begin{corollary}\label{c12} We have
\begin{enumerate}
    \item [\((i)\)] For \(n\geq 2\),
  \(DE_{e}(n+1) \equiv DE_{e}(n-2) \pmod{2}\) if and only if $n\ne T_j$,
    \item [\((ii)\)] For \(n\geq0\),
  \(DE_{e1}(n) \equiv DE_{e1}(n+1) \pmod{2}\) if and only if $n\ne T_j$.
\end{enumerate}
\end{corollary}

\begin{theorem}\label{t14}
   If \(k\in\mathbb{Z}\), \(n\neq k^2\) and \(2k^2\),   then
   \begin{enumerate}
       \item[\((i)\)] For \(n\geq j+2,\) \(\sum_{n \geq 0}DE_{e}(n-j+1)a(j)\equiv3\sum_{n\geq 0}DE_{e}(n-j-2)a(j)\pmod4\),
       \item[\((ii)\)] For \(n\geq j,\) \(\sum_{n \geq 0}DE_{e1}(n-j+1)a(j)\equiv3\sum_{n\geq 0}DE_{e1}(n-j)a(j)\pmod4\).
   \end{enumerate}
\end{theorem}
\begin{theorem}\label{t15}
     If \(k,l\in\mathbb{Z}\),  \(n\neq k^2,\) \(2k^2,\) \(4k^2,\) and \(k^2+2l^2\), then
    \begin{enumerate}
      \item[\((i)\)] For \(n\geq j+2,\) \(\sum_{n \geq 0}DE_{e}(n-j+1)a(j)\equiv3\sum_{n\geq 0}DE_{e}(n-j-2)a(j)\pmod8\),
      \item[\((ii)\)] For \(n\geq j,\)  \(\sum_{n \geq 0}DE_{e1}(n-j+1)a(j)\equiv3\sum_{n\geq 0}DE_{e1}(n-j)a(j)\pmod8\).
    \end{enumerate}
    \end{theorem}

\section{Proofs of Theorems \ref{t1} - \ref{t15}}

\noindent {\it Proof of the Theorem \ref{t1}:} We have
\begin{align}
\sum_{n \geq 0} \dfrac{\qPoch{q}{q}{2n+k-1} q^{2n}}{\qPoch{q^4}{q^4}{n}} 
&= \qPoch{q}{q}{\infty} \sum_{n \geq 0} \dfrac{q^{2n}}{\qPoch{q^4}{q^4}{n} \qPoch{q^{2n+k}}{q}{\infty}} \notag\\
&= \qPoch{q}{q}{\infty} \sum_{n \geq 0} \dfrac{q^{2n}}{\qPoch{q^4}{q^4}{n}} \sum_{m \geq 0} \dfrac{q^{(2n+k)m}}{\qPoch{q}{q}{m}} \notag\\
&= \qPoch{q}{q}{\infty} \sum_{m \geq 0} \dfrac{q^{km}}{\qPoch{q}{q}{m}} \sum_{n \geq 0} \dfrac{q^{(2m+2)n}}{\qPoch{q^4}{q^4}{n}} \notag\\
&= \qPoch{q}{q}{\infty} \sum_{m \geq 0} \dfrac{q^{km}}{\qPoch{q}{q}{m}} \dfrac{1}{\qPoch{q^{2m+2}}{q^4}{\infty}} \notag\\
&= \qPoch{q}{q}{\infty} \sum_{m \geq 0} \left( \dfrac{q^{2mk}}{\qPoch{q}{q}{2m}\qPoch{q^{4m+2}}{q^4}{\infty}} 
    + \dfrac{q^{(2m+1)k}}{\qPoch{q}{q}{2m+1}\qPoch{q^{4m+4}}{q^4}{\infty}} \right)\notag \\
&= \dfrac{\qPoch{q}{q}{\infty}}{\qPoch{q^2}{q^4}{\infty}} \sum_{m \geq 0} \frac{q^{2mk}\qPoch{q^2}{q^4}{m}}{\qPoch{q}{q}{2m}} 
    + \dfrac{\qPoch{q}{q}{\infty}}{\qPoch{q^4}{q^4}{\infty}} \sum_{m \geq 0} \dfrac{q^{(2m+1)k}\qPoch{q^4}{q^4}{m}}{\qPoch{q}{q}{2m+1}} \notag\\
&= \dfrac{\qPoch{q}{q}{\infty}}{\qPoch{q^2}{q^4}{\infty}} \sum_{m \geq 0} \dfrac{q^{2mk}\qPoch{-q}{q^2}{m}}{\qPoch{q^2}{q^2}{m}} 
    + \dfrac{\qPoch{q}{q}{\infty}}{\qPoch{q^4}{q^4}{\infty}} \sum_{m \geq 0} \dfrac{q^{(2m+1)k}\qPoch{-q^2}{q^2}{m}}{\qPoch{q}{q^2}{m+1}} \notag\\
&= \dfrac{\qPoch{q}{q}{\infty}}{\qPoch{q^2}{q^4}{\infty}}\cdot\dfrac{\qPoch{-q^{2k+1}}{q^2}{\infty}}{\qPoch{q^{2k}}{q^2}{\infty}} 
    + \dfrac{\qPoch{q}{q}{\infty}}{\qPoch{q^4}{q^4}{\infty}} \sum_{m \geq 0}DE_{\geq k}(m) q^m\notag\\
&\label{e1}= \dfrac{1}{1+q}\cdot\dfrac{\qPoch{q^2}{q^2}{k-1}}{\qPoch{-q^3}{q^2}{k-1}}  + \frac{\qPoch{q}{q}{\infty}}{\qPoch{q^4}{q^4}{\infty}} \sum_{m \geq 0}DE_{\geq k}(m) q^m.
\end{align}
Multiplying both sides of \eqref{e1} by $\dfrac{\qPoch{q^4}{q^4}{\infty}}{\qPoch{q}{q}{\infty}}$, we obtain
\begin{equation*}\label{r1}
\dfrac{\qPoch{q^4}{q^4}{\infty}}{\qPoch{q}{q}{\infty}} \sum_{n \geq 0}\dfrac{\qPoch{q}{q}{2n+k-1} q^{2n}}{\qPoch{q^4}{q^4}{n}} 
= \dfrac{\qPoch{q^4}{q^4}{\infty}}{(1+q)\qPoch{q}{q}{\infty}}\cdot \dfrac{\qPoch{q^2}{q^2}{k-1}}{\qPoch{-q^3}{q^2}{k-1}}  
    + \sum_{m \geq 0}DE_{\geq k}(m) q^m,
\end{equation*}
which is equivalent to
 \begin{align*}
\sum_{m \geq 0}DE_{\geq k}(m) q^m 
&= \dfrac{\qPoch{q^4}{q^4}{\infty}}{\qPoch{q}{q}{\infty}} \sum_{n \geq 0}\dfrac{\qPoch{q}{q}{2n+k-1} q^{2n}}{\qPoch{q^4}{q^4}{n}} - \dfrac{\qPoch{q^4}{q^4}{\infty}}{(1+q)\qPoch{q}{q}{\infty}}\cdot \dfrac{\qPoch{q^2}{q^2}{k-1}}{\qPoch{-q^3}{q^2}{k-1}} \\
&= \dfrac{\qPoch{q^4}{q^4}{\infty}}{\qPoch{q}{q}{\infty}} \sum_{n \geq 0}\dfrac{(q;q)_{k-1} (q^{k};q)_{2n}q^{2n}}{\qPoch{q^4}{q^4}{n}} - \dfrac{\qPoch{q^4}{q^4}{\infty}}{(1+q)\qPoch{q}{q}{\infty}}\cdot \dfrac{\qPoch{q^2}{q^2}{k-1}}{\qPoch{-q^3}{q^2}{k-1}} \\
&= \dfrac{\qPoch{q^4}{q^4}{\infty}(q;q)_{k-1} }{\qPoch{q}{q}{\infty}} \sum_{n \geq 0}\dfrac{(q^{k};q^2)_n (q^{k+1};q^2)_nq^{2n}}{\qPoch{q^4}{q^4}{n}} - \dfrac{\qPoch{q^4}{q^4}{\infty}}{(1+q)\qPoch{q}{q}{\infty}}\cdot \dfrac{\qPoch{q^2}{q^2}{k-1}}{\qPoch{-q^3}{q^2}{k-1}}. 
\end{align*}
Thus, the proof is complete.\\

\noindent{\it Proof of the Theorem \ref{t2}:} We have
\begin{align}
\sum_{n\geq 0} \dfrac{\qPoch{q}{q}{2n+k-1} q^{4n}}{\qPoch{q^4}{q^4}{n}} 
&= \qPoch{q}{q}{\infty} \sum_{n\geq 0} \dfrac{q^{4n}}{\qPoch{q^4}{q^4}{n} \qPoch{q^{2n+k}}{q}{\infty}}\notag \\
&= \qPoch{q}{q}{\infty} \sum_{n\geq 0} \dfrac{q^{4n}}{\qPoch{q^4}{q^4}{n}} 
   \sum_{m \geq 0} \dfrac{q^{2nm+mk}}{\qPoch{q}{q}{m}} \notag\\
&= \qPoch{q}{q}{\infty} \sum_{m\geq 0} \dfrac{q^{mk}}{\qPoch{q}{q}{m} \qPoch{q^{2m+4}}{q^4}{\infty}} \notag\\
&= \qPoch{q}{q}{\infty} \sum_{m\geq 0} \dfrac{q^{2mk}}{\qPoch{q}{q}{2m} \qPoch{q^{4m+4}}{q^4}{\infty}} 
   + \qPoch{q}{q}{\infty} \sum_{m\geq 0} \dfrac{q^{(2m+1)k}}{\qPoch{q}{q}{2m+1} \qPoch{q^{4m+6}}{q^4}{\infty}} \notag\\
&= \dfrac{\qPoch{q}{q}{\infty}}{\qPoch{q^4}{q^4}{\infty}} \sum_{m\geq 0} \dfrac{\qPoch{q^4}{q^4}{m} q^{2mk}}{\qPoch{q}{q}{2m}} 
   + \dfrac{\qPoch{q}{q}{\infty}}{\qPoch{q^6}{q^4}{\infty}} \sum_{m\geq 0} \dfrac{\qPoch{q^6}{q^4}{m} q^{(2m+1)k}}{\qPoch{q}{q}{2m+1}} \notag\\
&= \dfrac{\qPoch{q}{q}{\infty}}{\qPoch{q^4}{q^4}{\infty}} \sum_{m\geq 0} \dfrac{\qPoch{-q^2}{q^2}{m} q^{2mk}}{\qPoch{q}{q^2}{m}} 
   + \dfrac{q^k\qPoch{q}{q}{\infty}}{(1-q)\qPoch{q^6}{q^4}{\infty}} \sum_{m\geq 0} \dfrac{\qPoch{-q^3}{q^2}{m} q^{2mk}}{\qPoch{q^2}{q^2}{m}} \notag\\
&\label{e2}= \dfrac{\qPoch{q}{q}{\infty}}{\qPoch{q^4}{q^4}{\infty}} \sum_{m\geq 0} \dfrac{\qPoch{-q^2}{q^2}{m} q^{2mk}}{\qPoch{q}{q^2}{m}}
   + \dfrac{q^k\qPoch{q}{q}{\infty}}{(1-q)\qPoch{q^6}{q^4}{\infty}}\cdot \dfrac{\qPoch{-q^{3+2k}}{q^2}{\infty}}{\qPoch{q^{2k}}{q^2}{\infty}}. 
\end{align}
Multiplying both sides of \eqref{e2} by $\dfrac{q^k\qPoch{q^4}{q^4}{\infty}}{\qPoch{q}{q}{\infty}}$, we obtain
\begin{equation*}
\sum_{m\geq 0}DE_k(m)q^m
= \dfrac{q^k\qPoch{q^4}{q^4}{\infty}}{\qPoch{q}{q}{\infty}} \sum_{n\geq 0} \dfrac{\qPoch{q}{q}{2n+k-1} q^{4n}}{\qPoch{q^4}{q^4}{n}}-\dfrac{\qPoch{q^4}{q^4}{\infty}}{\qPoch{q}{q}{\infty}} \cdot \dfrac{q^{2k}\qPoch{q}{q}{\infty}}{(1-q)\qPoch{q^6}{q^4}{\infty}}\cdot \dfrac{\qPoch{-q^{3+2k}}{q^2}{\infty}}{\qPoch{q^{2k}}{q^2}{\infty}}.
\end{equation*}
Hence, the proof is complete.\\

\noindent{\it Proof of the Theorem \ref{t4}:} We have
\begin{align}
&\sum_{n \geq 0} \dfrac{\qPoch{q}{q}{2n+1}}{\qPoch{q^4}{q^4}{n}} q^{4n} \notag\\
&= \dfrac{1}{q^2} \left( \sum_{n \geq 0} \dfrac{\qPoch{q}{q}{2n+1}}{\qPoch{q^4}{q^4}{n}} q^{2n} - \sum_{n \geq 0} \dfrac{\qPoch{q}{q}{2n+2}}{\qPoch{q^4}{q^4}{n}} q^{2n}\right)
\notag\\
&= \dfrac{1}{q^2} \left(\sum_{n \geq 0} \dfrac{\qPoch{q^2}{q^2}{n} \qPoch{q^3}{q^2}{n}}{\qPoch{q^2}{q^2}{n} \qPoch{-q^2}{q^2}{n}} q^{2n} - \sum_{n \geq 0} \dfrac{\qPoch{q}{q}{2}\qPoch{q^3}{q^2}{n} \qPoch{q^4}{q^2}{n}}{\qPoch{q^2}{q^2}{n} \qPoch{-q^2}{q^2}{n}} q^{2n}\right)
\notag\\
&= \dfrac{1}{q^2} \left(\left( \sum_{n \geq 0} \dfrac{\qPoch{q^3}{q^2}{n}}{\qPoch{-q^2}{q^2}{n}} q^{2n} \right) \left( 1 - (1-q)(1-q^2) \right) + (1-q)(1-q^2)\left(\sum_{n \geq 0} \dfrac{\qPoch{q^3}{q^2}{n}}{\qPoch{-q^2}{q^2}{n}} q^{4n+2}\right) \right)
\notag\\
&= \dfrac{1}{q^2} \Big\{\left( \dfrac{1+q}{1+q^3}-\dfrac{\qPoch{q^3}{q^2}{\infty}}{\qPoch{-q^2}{q^2}{\infty}(1+q^3)}\right) \left( -q^3+q^2+q \right) \notag\\&\hspace{8cm}+ (1-q-q^2+q^3)\dfrac{q^2}{1-q}\left(\sum_{n \geq 0} \dfrac{\qPoch{q}{q}{2n+1}}{\qPoch{q^4}{q^4}{n}} q^{4n}\right) \Big\}
\notag\\
&\label{pt3e1}=\left(-\dfrac{1}{q}+\dfrac{1}{q^2}+\dfrac{1}{q^3}\right)\left(\dfrac{1+q}{1+q^3}-\dfrac{\qPoch{q^2}{q}{\infty}}{\qPoch{q^4}{q^4}{\infty}(1+q^3)}\right).
\end{align}
Employing  \eqref{t2e} in \eqref{pt3e1}, we obtain 
\begin{align}
\hspace{-1cm}\sum_{n\geq 0}DE_2(n)q^n&=q^2\left(-\dfrac{1}{q}+\dfrac{1}{q^2}+\dfrac{1}{q^3}\right)\left(\left(\dfrac{1+q}{1+q^3}\right)\dfrac{\qPoch{q^4}{q^4}{\infty}}{\qPoch{q}{q}{\infty}}-\dfrac{1}{(1+q^3)(1-q)}\right)
\notag\\
&\hspace{7cm}-\dfrac{q^4\qPoch{q}{q}{\infty}\qPoch{-q^7}{q^2}{\infty}\qPoch{q^4}{q^4}{\infty}}{(1-q)\qPoch{q}{q}{\infty}\qPoch{q^6}{q^4}{\infty}\qPoch{q^4}{q^2}{\infty}}
\notag\\
&\label{e4}=\dfrac{\qPoch{q^4}{q^4}{\infty}}{\qPoch{q}{q}{\infty}}\cdot\dfrac{q^2+2q^3-q^5+2q^8+q^9-q^{10}}{q^3(1+q^3)(1+q^5)}-\dfrac{q^2+q^3-q^5}{q^3(1+q^3)(1-q)}.
\end{align}Employing \eqref{regg} in \eqref{e4}, we obtain
$$\hspace{-5cm}q^3(1+q^3)(1+q^5)(1-q)\sum_{n\geq 0}DE_2(n)q^n=(1+q^5)(q^5-q^3-q^2)$$
\begin{equation}\label{e5}\hspace{3cm}+(1-q)(q^2+2q^3-q^5+2q^8+q^9-q^{10})\sum_{m\ge0}^{\infty}reg_4(n)q^{n}.
\end{equation}
Equating the coefficients of  \(q^n\) on both sides of \eqref{e5}, we arrive at the desired result.\\

\noindent{\it Proof of the Theorem \ref{t3}:} We have
\begin{align}\sum_{n \geq 0} \dfrac{\qPoch{q^3}{q^2}{n} \qPoch{q^4}{q^2}{n}}{\qPoch{q^4}{q^4}{n}} 
&= \dfrac{1}{(1-q^2)} \sum_{n \geq 0} \dfrac{\qPoch{q^3}{q^2}{n}}{\qPoch{-q^2}{q^2}{n}} \left( 1 - q^{2n+2} \right) q^{2n}\notag\\
&= \dfrac{1}{(1-q^2)} \left(\sum_{n \geq 0} \dfrac{\qPoch{q^3}{q^2}{n}}{\qPoch{-q^2}{q^2}{n}} q^{2n} - \sum_{n \geq 0} \dfrac{\qPoch{q^3}{q^2}{n}}{\qPoch{-q^2}{q^2}{n}} q^{4n+2} \right)
\notag\\
&\label{e6}= \dfrac{1}{(1-q^2)} \left(\sum_{n \geq 0} \dfrac{\qPoch{q^3}{q^2}{n}}{\qPoch{-q^2}{q^2}{n}} q^{2n} - \dfrac{q^2}{1-q} \sum_{n \geq 0} \dfrac{\qPoch{q}{q}{2n+1}}{\qPoch{q^4}{q^4}{n}} q^{4n} \right).
\end{align}
Employing \eqref{pt3e1} in \eqref{e6}, we obtain
$$\hspace{-5cm}
    \sum_{n \geq 0} \dfrac{\qPoch{q^3}{q^2}{n} \qPoch{q^4}{q^2}{n}}{\qPoch{q^4}{q^4}{n}} = \dfrac{1}{(1-q^2)} \Big\{-\dfrac{\qPoch{q^3}{q^2}{\infty}}{\qPoch{-q^2}{q^2}{\infty} (1+q^3)} + \dfrac{1+q}{1+q^3}$$\begin{equation}\label{pt4e1} \hspace{2.5cm}- \dfrac{q^2}{1-q}\left( -\dfrac{1}{q} + \dfrac{1}{q^2} + \dfrac{1}{q^3} \right)
    \left( \dfrac{1+q}{1+q^3}- \dfrac{1}{1+q^3} \cdot \dfrac{\qPoch{q^2}{q}{\infty}}{\qPoch{q^4}{q^4}{\infty}}\right)\Big\}.
\end{equation}
Employing \eqref{pt4e1} in  \eqref{t1e} and simplifying, we obtain
\[
\sum_{n\geq 0} DE_{\geq 3}(n) q^n= -\dfrac{1}{1+q^3} + \dfrac{1+q}{1+q^2} \cdot\dfrac{\qPoch{q^4}{q^4}{\infty}}{\qPoch{q^2}{q}{\infty}}- \dfrac{\qPoch{q^4}{q^4}{\infty}}{\qPoch{q}{q}{\infty}} \left( \dfrac{1}{q} + 1 - q \right) \dfrac{(1+q)}{(1+q^3)}
+ \dfrac{\left(\dfrac{1}{q} + 1 - q\right)}{(1+q^3)(1-q)} 
\]\[\hspace{4cm}
-
\dfrac{\qPoch{q^4}{q^4}{\infty}}{\qPoch{q}{q}{\infty}} \cdot \dfrac{(1-q)(1-q^4)}{(1+q^3)(1+q^5)}
\]
\begin{equation}\label{e7}\hspace{3cm}=\dfrac{\qPoch{q^4}{q^4}{\infty}}{\qPoch{q}{q}{\infty}}\left(\dfrac{-1-2q+q^2-q^5-q^7}{q(1+q^3)(1+q^5)}\right)-\dfrac{2}{1+q^3}+\dfrac{1}{q(1-q)(1+q^3)}.
\end{equation}
Employing \eqref{regg} in \eqref{e7} and simplifying, we obtain
\begin{equation}\label{e9}
q(1-q)(1+q^3)(1+q^5)\sum_{n \geq 0} DE_{\geq 3}(n) q^n =(-1-3q+3q^2-q^5+q^6)\sum_{m\ge0}^{\infty}reg_4(n)q^{n}-2(q-q^2-q^7)+1-q^5.
\end{equation}
Equating the coefficients of \(q^n\) on both sides of \eqref{e9}, we arrive at the desired result.\\

\noindent{\it Proof of the Theorem \ref{t5}:}
From  \eqref{ie9} and \eqref{d3e}, we obtain
\begin{equation}\label{pt5e1}
  \sum_{n \geq 0} DE_e(n) q^n = \dfrac{1}{q} \sum_{n \geq 0} \operatorname{DE3}(n) q^n.  
\end{equation}
Multiplying both sides of  \eqref{pt5e1} by \((1+q^3)\), we obtain 
\begin{equation}\label{pt5e2}
   (1+q^3) \sum_{n \geq 0} DE_e(n) q^n = \dfrac{(1+q^3)}{q} \sum_{n \geq 0} \operatorname{DE3}(n) q^n. 
\end{equation}
Now from \cite[Theorem 3, (2.6)]{DE}, we note that 
\begin{equation}\label{an1}(1+q^3)\sum_{n \geq 0} \operatorname{DE3}(n) q^n=q^2\dfrac{\qPoch{q^4}{q^4}{\infty}}{\qPoch{q}{q}{\infty}}-q^2+q.\end{equation}
Employing \eqref{an1} in \eqref{pt5e2}, we obtain
\[(1+q^3) \sum_{n \geq 0} DE_e(n) q^n = q\dfrac{\qPoch{q^4}{q^4}{\infty}}{\qPoch{q}{q}{\infty}} - q + 1.\]
Employing \eqref{regg} in \eqref{an1} and simplifying, we obtain
\begin{equation}\label{pt5e3}
 \sum_{n \geq 0} DE_e(n) q^n = \sum_{n \geq 0} reg_4(n) q^{n+1}-q+1- \sum_{n \geq 0} DE_e(n) q^{n+3} .   
\end{equation}
Extracting the coefficients of $q^n$ on both sides of \eqref{pt5e3}, we obtain
\begin{equation}\label{pt5e4}
 DE_e(n) = reg_4(n-1)- DE_e(n-3) , \quad \text{for } n \geq 3,   
\end{equation}
Iterating \eqref{pt5e4} in the following manner, we obtain
    \begin{align*}
        DE_e(n) &=reg_4(n-1)- DE_e(n-3)\\
                &=reg_4(n-1)-reg_4(n-4)+DE_e(n-6)\\
                &=reg_4(n-1)-reg_4(n-4)+reg_4(n-7)- DE_e(n-9)\\
                &\;\;\vdots\\
                &=\sum_{i=0}^{\left\lfloor \frac{n-1}{3} \right\rfloor} (-1)^ireg_4(n-1-3i).
    \end{align*}\\
Hence, the proof is complete.\\

\noindent{\it Proof of the Theorem \ref{t6}:}
From \eqref{ie7} and \eqref{d4e}  we obtain
\begin{equation}\label{pt6e1}
   \sum_{n \geq 0} DE_{e1}(n) q^n = q \sum_{n \geq 0} \operatorname{DE1}(n) q^n. 
\end{equation}
Multiplying both sides of  \eqref{pt6e1} by \((1+q)\), we obtain 
\begin{equation}\label{pt6e2}
   \dfrac{(1+q)}{q} \sum_{n \geq 0} DE_{e1}(n) q^n = (1+q) \sum_{n \geq 0} \operatorname{DE1}(n) q^n. 
\end{equation}
From \cite[Theorem 1 (2.2)]{DE}, we note that
\begin{equation}\label{an2}
	(1+q)\sum_{n \geq 0} \operatorname{DE1}(n) q^n=\dfrac{\qPoch{q^4}{q^4}{\infty}}{\qPoch{q}{q}{\infty}}-1
	\end{equation}
Employing \eqref{an2} in \eqref{pt6e2}, iwe obtain
\begin{equation}\label{pt6e3}
 \dfrac{(1+q)}{q} \sum_{n \geq 0} DE_{e1}(n) q^n = \sum_{n \geq 0} reg_4(n) q^{n}-1.   
\end{equation}
Equating the coefficients of $q^n$ on both sides of equation \eqref{pt6e3} and simplifying, we obtain 
\begin{equation}\label{pt6e4}
    DE_{e1}(n+1) = reg_4(n)-+ DE_{e1}(n), \quad \text{for } n \geq 1.   
\end{equation}
Iterating \eqref{pt6e4}, we obtain
 \begin{align}
        DE_{e1}(n+1)&=reg_4(n)- DE_{e1}(n)\notag\\
                    &=reg_4(n)- reg_4(n-1)+DE_{e1}(n-1)\notag\\
                    &=reg_4(n)- reg_4(n-1)+reg_4(n-2)- DE_{e1}(n-2)\notag\\
                    &\;\;\vdots\notag\\
                    &\label{e10}=\sum_{i=0}^{n-1} (-1)^ireg_4(n-i).
        \end{align}
  Replacing \(n\) by \(n-1\) in the \eqref{e10}, we arrive at the desired result.\\
    
\noindent{\it Proof of the Theorem \ref{t7}:}
It is easily seen that
\begin{equation}\label{pt7e1}
  \sum_{m\geq 0}\dfrac{\qPoch{-q^2}{q^2}{m}}{\qPoch{q}{q^2}{m}}q^{4m}
=\dfrac{1}{q^2}\left(\sum_{m\geq 0}\dfrac{\qPoch{-q^2}{q^2}{m}}{\qPoch{q}{q^2}{m+1}}q^{4m+2}-\sum_{m\geq 0}\dfrac{\qPoch{-q^2}{q^2}{m}}{\qPoch{q}{q^2}{m+1}}q^{6m+3}\right) 
\end{equation}
Employing \eqref{d5e} in \eqref{pt7e1}, we obtain
\begin{equation}\label{pt7e2}
  DE_{e\geq2}(n-2)=DE_{\geq2}(n)-DE_{\geq3}(n). 
\end{equation}
Now, \cite[Corollary 2]{DE} and using the fact that $DE2(n)=DE_{\ge2}(n)$,  we note that
\begin{equation}\label{kj} DE_{\geq2}(n)=reg_{4>1}(n)-DE_{\geq2}(n-3),
	\end{equation}where $reg_{4>1}(n)$ denotes the number of $4$-regular partitions of $n$ into parts each$>1$.\\
Employing \eqref{kj} in  (\ref{pt7e2}) and iterating, we obtain
\begin{align*}
DE_{e\geq2}(n-2)+DE_{\geq3}(n)&=DE_{\geq2}(n)\\
       &=reg_{4>1}(n)-DE_{\geq2}(n-3)\\
       &=reg_{4>1}(n)-reg_{4>1}(n-3)+DE_{\geq2}(n-6)\\
       &\;\;\vdots\\
       &=\sum_{i=0}^{n-1} (-1)^ireg_{4>1}(n-3i).
\end{align*}
Hence, the proof is complete.\\

\noindent{\it Proof of the Theorem \ref{t8}:} For any complex number $z\ne0$ and $|q|<1$, the Jacobi's triple product identity \cite[Theorem 11]{jacobi1} is given by 
\begin{equation}\label{pt89e1}
   \qPoch{q^2}{q^2}{\infty}\qPoch{-qz}{q^2}{\infty}\qPoch{-q/z}{q^2}{\infty}= \sum_{n=-\infty}^{\infty} z^n q^{n^2},.
\end{equation}
By replacing $q$ by $q^2$ and $z$ by $-q$ in \eqref{pt89e1}, we obtain
\begin{equation}\label{pt89e2}
    \qPoch{q^4}{q^4}{\infty}\qPoch{q^3}{q^4}{\infty}\qPoch{q}{q^4}{\infty}= \sum_{n=-\infty}^{\infty} (-1)^n q^{2n^2+n},
\end{equation}
which implies,
\begin{equation}\label{pt89e3}
\dfrac{\qPoch{q^4}{q^4}{\infty}^2}{\qPoch{q^2}{q^2}{\infty}} = \left(\sum_{n \geq 0} reg_4(n) q^n\right)\left(\sum_{n=0}^{\infty} (-1)^{\lceil \frac{n}{2} \rceil} q^{\frac{n(n+1)}{2}}\right),  
\end{equation}where $\lceil \cdot\rceil$ denotes the ceiling function. 
Again, in terms of Ramanujan theta function, the Jacobi's triple product identity is expressed as
\begin{equation}\label{pt89e4}
	\qPoch{-q}{qz}{\infty}\qPoch{-z}{zq}{\infty}\qPoch{zq}{zq}{\infty}= \sum_{n=-\infty}^{\infty} q^{\frac{n(n+1)}{2}} z^{\frac{n(n-1)}{2}}, \quad |qz| < 1. 
\end{equation}
Replacing $q$ by $q^2$ and $z$ by $1$ is \eqref{pt89e4} and simplifying, we obtain
\begin{equation}\label{pt89e5}
  \dfrac{\qPoch{q^4}{q^4}{\infty}^2}{\qPoch{q^2}{q^2}{\infty}} = \sum_{n=0}^{\infty} q^{n(n+1)}.  
\end{equation}
Employing \eqref{pt89e3} and \eqref{pt89e5}, we obtain
\begin{equation}\label{pt89e6}
\left(\sum_{n=0}^{\infty} reg_4(n) q^n \right) \left(\sum_{n=0}^{\infty} (-1)^{\lceil \frac{n}{2} \rceil} q^{\frac{n(n+1)}{2}} \right) = \sum_{n=0}^{\infty} q^{n(n+1)} . 
\end{equation}
Combining \eqref{pt5e4}, \eqref{pt6e4}, and \ref{pt89e6}, we deduce that
\begin{equation}\label{pt89e7}
\left(reg_4(0)+reg_4(1)+\sum_{n=2}^{\infty}\left( DE_e(n+1)+DE_e(n-2)\right) q^n \right) \left(\sum_{n=0}^{\infty} (-1)^{\lceil \frac{n}{2} \rceil} q^{\frac{n(n+1)}{2}} \right) = \sum_{n=0}^{\infty} q^{n(n+1)},  
\end{equation}  and
\begin{equation}\label{pt89e8}
 \left(reg_4(0)+\sum_{n=1}^{\infty}\left( DE_{e1}(n+1)+DE_{e1}(n)\right) q^n \right) \left(\sum_{n=0}^{\infty} (-1)^{\lceil \frac{n}{2} \rceil} q^{\frac{n(n+1)}{2}} \right) = \sum_{n=0}^{\infty} q^{n(n+1)}.
\end{equation}
Comparing the coefficients of \(q^n\) in (\ref{pt89e7}), and (\ref{pt89e8}), we complete the proof of  Theorems \ref{t8} (i) and (ii), respectively.\\

\noindent{\it Proof of the Theorem \ref{t10}:} Replacing \(q\) by \(-q^2\) and \(z\) by \(-q^2\) in \eqref{pt89e4}, we obtain
\begin{equation}\label{pt1011e1}
   \qPoch{q^2}{q^4}{\infty} \qPoch{q^2}{q^4}{\infty} \qPoch{q^4}{q^4}{\infty} = \sum_{n=-\infty}^{\infty} (-1)^n q^{2n^2}. 
\end{equation}
Simplifying \eqref{pt89e4} using \eqref{regg}, we obtain
\begin{equation}\label{pt1011e2}
\dfrac{ \qPoch{q^2}{q^2}{\infty}^2}{ \qPoch{q}{q}{\infty}} = \left(\sum_{n=0}^{\infty} reg_4(n) q^n\right) \left(\sum_{n=-\infty}^{\infty} (-1)^n q^{2n^2} \right). 
\end{equation}
Again, replacing \(z\) by \(1\) in \eqref{pt89e4} and simplifying, we obtain
\begin{equation}\label{pt1011e3}
\dfrac{ \qPoch{q^2}{q^2}{\infty}^2}{ \qPoch{q}{q}{\infty}}  = \sum_{n=0}^{\infty} q^{\frac{n(n+1)}{2}}.
\end{equation}
Employing \eqref{pt1011e3} in \eqref{pt1011e2}, we obtain
\begin{equation}\label{pt1011e4}
\left(\sum_{n=0}^{\infty} reg_4(n) q^n \right)\left(\sum_{n=-\infty}^{\infty} (-1)^n q^{2n^2}\right) = \sum_{n=0}^{\infty} q^{\frac{n(n+1)}{2}}.
\end{equation}
Combining \eqref{pt5e4}, \eqref{pt6e4}, and \eqref{pt1011e4}, we obtain
\begin{equation}\label{pt1011e5}
\left(reg_4(0)+reg_4(1)+\sum_{n=2}^{\infty}\left( DE_e(n+1)+DE_e(n-2)\right) q^n \right) \left(\sum_{n=-\infty}^{\infty} (-1)^n q^{2n^2}\right) = \sum_{n=0}^{\infty} q^{\frac{n(n+1)}{2}},  
\end{equation} and
\begin{equation}\label{pt1011e6}
\left(reg_4(0)+\sum_{n=1}^{\infty}\left( DE_{e1}(n+1)+DE_{e1}(n)\right) q^n \right) \left(\sum_{n=-\infty}^{\infty} (-1)^n q^{2n^2}\right) =\sum_{n=0}^{\infty} q^{\frac{n(n+1)}{2}}.
\end{equation}
Comparing the coefficients of \(q^n\) in both sides of \eqref{pt1011e5} and \eqref{pt1011e6}, we complete the proof of Theorem \ref{t10}(i) and (ii), respectively.\\

\noindent{\it Proof of the Theorem \ref{t14}:}
Set,
\[
\mathrm{A}(q) := \left(\sum_{n\ge0}reg_4(n)q^n\right)\left(\sum_{m\ge0}a(m)q^m\right)=\dfrac{\qPoch{q^4}{q^4}{\infty}}{\qPoch{q}{q}{\infty}^2\qPoch{q^2}{q^2}{\infty}}.
\]
Then it is easy to see that, 
\begin{equation}\label{he1}
\mathrm{A}(q) = \varphi(q)\varphi(q^2) \mathrm{A}(q^2)^2,
\end{equation}
where
\[
\varphi(q) := \sum_{k=-\infty}^{\infty} q^{k^2} 
= \dfrac{\qPoch{q^2}{q^2}{\infty}^5}{\qPoch{q}{q}{\infty}^2\qPoch{q^4}{q^4}{\infty}^2}.
\]
Iterating \eqref{he1} for infinite times, we obtain
\begin{align*}
    \mathrm{A}(q) &= \varphi(q) \varphi(q^2) \mathrm{A}(q^2)^2\\
     &=\varphi(q)\varphi(q^2) 
\bigl(\varphi(q^2)\varphi(q^4) \mathrm{A}(q^4)^2 \bigr)^2\\ 
&=\varphi(q)\varphi(q^2)^{3}\varphi(q^4)^{2} \mathrm{A}(q^4)^4\\
        &=\varphi(q)\varphi(q^2)^3\varphi(q^4)^{2}
\bigl(\varphi(q^4)\varphi(q^8) \mathrm{A}(q^8)^2 \bigr)^4\\
&=\varphi(q)\varphi(q^2)^3\varphi(q^4)^{6}\varphi(q^8)^{4} \mathrm{A}(q^8)^8\\
&\;\;\vdots\\
&= \varphi(q) \prod_{i\ge1} \varphi(q^{2^i})^{3\cdot2^{i-1}}.
\end{align*}
By repeating the process , we obtain
Again, from the Binomial theorem, it is clear that
\(
\varphi(q^{2i})^{3\cdot 2^{i-1}} \equiv 1 \pmod{4}
\)
for each $i \ge 2$. Thus,
\begin{align*}
   \mathrm{A}(q)
&= \varphi(q) \prod_{i\ge1} \varphi(q^{2^i})^{3\cdot2^{i-1}}\\
&\equiv\varphi(q)\varphi(q^2)^{3} \pmod{4}\\
&=\left(1 + 2 \sum_{n\ge1} q^{n^2}\right)
\left(1 + 2 \sum_{n\ge1} q^{2n^2}\right)^{3}\\
&=\left(1 + 2 \sum_{n\ge1} q^{n^2}\right)
\left( \sum_{k=0}^{3} {3 \choose k} 2^k 
\left(\sum_{n\ge1} q^{2n^2}\right)^k \right)\\
&= \left(1 + 2 \sum_{n\ge1} q^{n^2}\right)
\left(1 + 6 \sum_{n\ge1} q^{2n^2}\right)
\pmod{4}.
\end{align*}
Therefore, 
\begin{equation}\label{pt14e1}
  \left(\sum_{n\ge0}reg_4(n)q^n\right)\left(\sum_{m\ge0}a(m)q^m\right)\equiv 1 + 2 \sum_{n\ge1} q^{n^2}
+ 6 \sum_{n\ge1} q^{2n^2}
\pmod{4}. 
\end{equation}
Emplyoing equations \eqref{pt5e4} and (\ref{pt6e4}) in \eqref{pt14e1}, and then comparing coefficients of \(q^n\),  we arrive at the Theorem \ref{t14}  (i) and (ii), respectively.\\

\noindent{\it Proof of the Theorem \ref{t15}:}
From the Binomial theorem, it is clear that \(\varphi(q^{2i})^{3\cdot 2^{i-1}} \equiv 1 \pmod{8}\) for each \(i \geq 3\).
Thus,
\begin{align*}
\mathrm{A}(q)&= \varphi(q) \prod_{i \geq 1} \varphi(q^{2^i})^{3\cdot2^{i-1}} \\
&\equiv \varphi(q)\varphi(q^2)^{3}\varphi(q^4)^{6} \pmod{8} \\
&= \left(1 + 2\sum_{n \geq 1} q^{n^2}\right) \left(1 + 2\sum_{n \geq 1} q^{2n^2}\right)^{3} \left(1 + 2\sum_{n \geq 1} q^{4n^2}\right)^{6} \\
&= \left(1 + 2\sum_{n \geq 1} q^{n^2}\right) \left(\sum_{k=0}^{3} \binom{3}{k} 2^k \left(\sum_{n \geq 1} q^{2n^2}\right)^k\right)\left(\sum_{k=0}^{6} \binom{6}{k} 2^k \left(\sum_{n \geq 1} q^{4n^2}\right)^k\right) \\
&\equiv \left(1 + 2\sum_{n \geq 1} q^{n^2}\right) \left(1 + 6 \sum_{n \geq 1} q^{2n^2} + 12 \sum_{n \geq 1} q^{4n^2}\right)\;\;\;\;\left(1 + 12 \sum_{n \geq 1} q^{4n^2} + 12 \sum_{n \geq 1} q^{8n^2}\right) \pmod{8}. 
\end{align*}
Therefore, 
\begin{align}\label{pt15e1}
\left(\sum_{n\ge0}reg_4(n)q^n\right)\left(\sum_{m\ge0}a(m)q^m\right)
&\equiv1 + 2\sum_{n \geq 1} q^{n^2} + 6 \sum_{n \geq 1} q^{2(2n-1)^2}+ 18 \sum_{n \geq 1} q^{2(2n)^2}\quad\\
& \label{ff}\qquad+ 24 \sum_{n \geq 1} q^{4n^2} + 12\sum_{m,n \geq 1} q^{n^2+2m^2} \pmod{8}. 
\end{align}
Emplyoing  \eqref{pt5e4} and \eqref{pt6e4} in \eqref{ff}, and then comparing coefficients of \(q^n\), we arrive at Theorem \ref{t15} (i) and (ii), respectively.\\

\section{Concluding Remarks}
 We  established new \(q\)-series identities involving partition functions of \(DE_{\geq k}(n)\) and \(DE_{k}(n)\) and relations for \(DE_{\geq 3}(n)\), \(DE_{2}(n)\), \(DE_{e}(n)\), \(DE_{e1}(n)\) and \(DE_{e\geq 2}(n)\) with \(reg_4(n)\). Some congruence relations of \(DE_{e}(n)\), and \(DE_{e1}(n)\) are also proved for modulo \(2,\) \(4,\) and \(8.\) Moreover,  from \eqref{s2e1}, \eqref{s2e2}, \eqref{d4e}, and \eqref{d5e} it can be easily seen that \(DE_{ek}(n+k)=DE_{\geq k}(n)\) and \(DE_{e\geq k}(n-k)=DE_{k}(n)\). Interested readers may try to find new relations for \(DE_{e\geq k}\) and \(DE_{ek}(n)\) and possible generalisations.  
 
   \section*{\bf Declarations}
   
   \noindent{\bf Author Contributions.} Both authors contributed equally to this work.
   
 \noindent{\bf Conflict of Interest.} The authors declare that there is no conflict of interest regarding the publication of this article.
 
 \noindent{\bf Human and animal rights.} The authors declare that there is no research involving human participants or animals in the contained of this paper.	
 
 \noindent{\bf Data availability statements.} Data sharing not applicable to this article as no datasets were generated or analysed during the current study.

\end{document}